\documentclass[12pt,a4paper]{amsart}
\usepackage{a4wide,amsfonts,amsmath,latexsym,amssymb,euscript,graphicx,dsfont,
units,color,amsthm}

\newcommand{\p}{\mathds{P}}
\newcommand{\e}{\mathds{E}}
\newcommand{\q}{\mathds{Q}}

\newcommand{\cf}{\mathcal F}

\newcommand{\gga}{\Gamma}
\newcommand{\ka}{\kappa}

\newcommand{\laa}{\Lambda}

\newcommand{\oom}{\Omega}

\newcommand{\lc}{\left[}
\newcommand{\rc}{\right]}

\newtheorem{thm}{Theorem}[section]

\newtheorem{prop}{Proposition}[section]
\newtheorem{lem}[prop]{Lemma}

\newtheorem{rem}[prop]{Remark}

\begin{document}

\title{On homogeneous pinning models and penalizations}
\author{Mihai Gradinaru \and Samy Tindel}



\address{
{\it Mihai Gradinaru:}
{\rm IRMAR, Universit\'e de Rennes 1, Campus de Beaulieu,
35042 Rennes Cedex.
}
\newline
$\mbox{ }$\hspace{0.1cm}
{\it Email: }{\tt Mihai.Gradinaru@univ-rennes1.fr}
\newline
$\mbox{ }$\hspace{0.1cm}
\newline
$\mbox{ }$\hspace{0.1cm}
{\it Samy Tindel:}
{\rm Institut {\'E}lie Cartan Nancy, B.P. 239,
54506 Vand{\oe}uvre-l{\`e}s-Nancy Cedex, France}.
\newline
$\mbox{ }$\hspace{0.1cm}
{\it Email: }{\tt tindel@iecn.u-nancy.fr}
}

\begin{abstract}
In this note, we show how the penalization method, introduced in order to
describe some non-trivial changes of the Wiener measure, can be applied
to the study of some simple polymer models such as the pinning model.
The bulk of the analysis is then focused on the study of a martingale which
has to be computed as a Markovian limit. 
\end{abstract}

\maketitle

%\begin{center}
%{\it Dedicated to Ludwig Arnold on occasion of his 70th birthday}
%\end{center}

%%%%%%%%%%%%%%%%%%%%%%%%%%%%%%%%%%%%%%%%%%
%%%%%%%%%%%%%%%%%%%%%%%%%%%%%%%%%%%%%%%%%%

\section{Introduction}~~
Our motivation for writing the current note is the following: on the one hand,
in the last past years, some interesting advances have seen the light concerning
various kind of polymer models, having either an interaction with a random environment or
a kind of intrinsic self-interaction. Among this wide class of models, we will be interested
here in some polymers interacting with a given interface, as developed for instance
in \cite{BdH,MGO}. For this kind of polymers, the introduction of some generalized renewal
tools has yield some very substantial progresses in the analysis of the model, and a quite
complete picture of their asymptotic behaviour in terms of localization near the
interface is now available e.g. in
\cite{CGZ,GT} and in the monograph \cite{Gi}.

\vspace{0.3cm}

On the other hand, and a priori in a different context, the series of papers starting by
\cite{RVY} and ending with the recent monograph \cite{RY} presents a rather simple
method in order to quantify the penalization of a Brownian (or Bessel) path by a functional
of its trajectory (such as the one-sided supremum or the number of excursions).
This method can then be applied in a wide number of natural situations, getting a very complete
description of some Gibbs type measures based on the original Brownian motion.
More specifically, when translated in a random walk context, the penalization method
can be read as follows:
let $\{b_n;\, n\ge 0 \}$ be a symmetric random walk on $\mathds{Z}$,
defined on a stochastic basis $(\oom,\cf,(\cf_n)_{n\ge 1},(\p_z)_{z\in\mathds{Z}})$.
%When this notation does not lead to any ambiguity, we will simply write$\p$ for $\p_0$. 
For $n\ge 0$, let also $e^{H_n}$ be a bounded positive
measurable functional of the path $(b_0,\ldots,b_n)$. 
Then, for $\beta\in\mathds{R}$, $n\ge p \ge 0$, we are concerned with a generic Gibbs type measure
$\rho_n$ on $\cf_p$ defined, for $\gga_p\in\cf_p$, by
\begin{equation}\label{def:hat-rho-n}
\rho_n(\gga_p)
=\frac{\e_{0}\left[\mathds{1}_{\gga_p} e^{\beta H_n} \right]}{Z_{n}} ,
\quad\mbox{ where }\quad
Z_{n}= \e_{0}\left[e^{\beta H_n}\right].
\end{equation}
In its general formulation, the penalization principle, which allows an asymptotic study 
of $\rho_n$, can be stated as follows:
\begin{prop}\label{prop:general-penalization}
Suppose that the process $(b_n,H_n)$ is a $\mathds{Z}\times\mathds{R}_+$-valued 
Markov process, and let $\laa_n$ be its semi-group.
Assume that, for any $p\ge 0$, the function $M_p$ defined by
\begin{equation}\label{markov:lim}
M_p(w,z):=\lim_{n\to\infty}
\frac{[\laa_{n-p}f](w,z)}{[\laa_n f](0)},
\quad\mbox{ where }\quad
f(w,z)=e^{-\beta z}
\end{equation}
exists, for any $(w,z)\in \mathds{Z}\times\mathds{R}_+$, and that
\begin{equation*}
\frac{[\laa_{n-p}f](w,z)}{[\laa_n f](0)}
\le
C(p,w,z),
\quad\mbox{ where }\quad
\e_{0}[C(p,b_p,\ell_p)]<\infty.
\end{equation*}
Then:
\begin{enumerate}
\item
the process $M_p:=M_p(b_p,\ell_p)$ is a non-negative $\p_{0}$-martingale;
\item
for any $p\ge 0$, when $n\to\infty$, the measure $\rho_n$ defined by (\ref{def:hat-rho-n})
converges weakly on $\cf_p$ to a measure $\rho$, where $\rho$ is defined
by
\begin{equation*}
\rho(\gga_p)=\e_{0}\left[\mathds{1}_{\gga_p}M_p\right],
\quad\mbox{ for }\quad
\gga_p\in\cf_p.
\end{equation*}
\end{enumerate}
\end{prop}
This last proposition can be seen then as an invitation 
to organize the asymptotic study of the measure $\rho_n$
in the following way: first compute explicitly the limit of the ratio
$[\laa_{p-n}f](w,z)/[\laa_p f](0)$ when $p\to\infty$, which should define also an asymptotic 
measure $\rho$ in the infinite volume regime. Then try to read the basic properties of $\rho$
by taking advantage of some simple relations on the martingale $M_p$.

\vspace{0.3cm}

It is easily seen that some links exists between the polymer measure theory as mentioned above
and the penalization method. Furthermore, we believe that the two theories can interact in a
fruitful way. Indeed, the penalizing scheme offers a simple and systematic framework for
the study of Gibbs measures based on paths, and it is also quite pleasant to be able to read the
main features of the limiting measure $\rho$ on the martingale  $M_p$, which is usually a simple
object. Without presenting a completely new result, this article will thus try to make a bridge 
between the two aspects of the topic, by studying the simplest of the interface-based polymers,
namely the polymer pinned at an interface, through a purely penalizing scheme. Let us be
more specific once again, and describe our model and the main results we shall obtain:
denote by $\ell_n$ the local time at 0 of $b$, that is
$$
\ell_n=\sharp \{p\le n;\,b_p=0\}.
$$
For $\beta\in\mathds{R}$, $n\ge p \ge 0$, we are concerned here with the Gibbs type measure
$\q_{0}^{(n,\beta)}$ on $\cf_p$ defined, for $\gga_p\in\cf_p$, $p<n$, by
\begin{equation}\label{def:rho-n}
\q_{0}^{(n,\beta)}(\gga_p)
=\frac{\e_{0}\left[\mathds{1}_{\gga_p}e^{\beta \ell_n} \right]}{Z_{n}^{f}} ,
\quad\mbox{ where }\quad
Z_{n}^{f}= \e_{0}\left[e^{\beta\ell_n}\right].
\end{equation}
Finally, we will need to introduce a slight variation of the Bessel walk of dimension 3,
which is defined as a random walk $R$ on $\mathds{N}$ starting from 0, such that
$\p_{0}(R_0=0)=\p_{0}(R_1=1)=1$, and whenever $j\ge 1$,
\begin{equation}\label{eq:transition-bessel-walk}
\p_{0}(R_{n+1}=j\pm 1\, |\, R_{n}=j)= \frac{j\pm 1}{2j}.
\end{equation}
With these notations in hand, the main result we shall obtain is then the following:
\begin{thm}\label{thm:desription-gibbs-limit}
For $\beta\in\mathds{R}$, $n\ge p \ge 0$, let $\q_{0}^{(n,\beta)}$ be the measure defined by 
(\ref{def:rho-n}). Then, for any $p\ge 0$, the measure $\q_{0}^{(n,\beta)}$ on $\cf_p$ converges
weakly, as $n\to\infty$, to a measure $\q_{0}^{(\beta)}$ defined by
\begin{equation}\label{def:rho}
\q_{0}^{(\beta)}(\gga_p)=\e_{0}\left[\mathds{1}_{\gga_p}M_p^{(\beta)}\right],
\quad\mbox{ for }\quad
\gga_p\in\cf_p.
\end{equation}
%Notice that the dependence on $\beta$ in the martingale $M_p^{(\beta)}$ will generally be omitted, in spite of its importance. 
According to the sign of $\beta$ the two following
situations can occur:

\vspace{0.3cm}

\noindent
{\bf (1)} When $\beta<0$ (delocalized phase): set $\alpha=-\beta$. Then $M_p^{(\beta)}$ has
the following expression:
$$
M_p^{(\beta)}=e^{-\alpha\ell_p}\lc (1-e^{-\alpha})|b_p| + 1 \rc.
$$
Moreover, under the probability $\q_{0}^{(\beta)}$, the process $b$ and its local time $\ell$
can be described in the following way:
\begin{enumerate}
\item[a)]
The random variable $\ell_{\infty}$ is finite almost surely, and is distributed
according to a geometric law with parameter $1-e^{-\alpha}$. 
\item[b)]
Let $g=\sup\{r\ge 0;\, b_r=0\}$.Then $g$ is finite almost surely, and the two processes
$b^{(-)}=\{b_r;\, r\le g\}$ and $b^{(+)}=\{b_{r+g};\, r\ge 0\}$ are independent. 
\item[c)]
The process $|b^{(+)}|$ is a Bessel random walk as defined by the transition law 
{\rm (\ref{eq:transition-bessel-walk})}, and $\mbox{sign}(b^{(+)})=\pm 1$ with probability $1/2$. 
\item[d)]
Given the event $\ell_{\infty}=l$ for $l\ge 1$, the process $b^{(-)}$ is a standard random
walk, stopped when its local time reaches $l$.
\end{enumerate}

\vspace{0.3cm}

\noindent
{\bf (2)} When $\beta>0$ (localized phase): in this case, the martingale $M_p^{(\beta)}$ 
can be written as:
\begin{equation}\label{eq:def-martingale-localized}
M_p^{(\beta)}=\exp\left\{\beta{\hat\ell}_{p}-c_{+,\beta}\,|b_{p}|-c_{-,\beta}\,p\right\},
\end{equation}
where $c_{\pm,\beta}=(\nicefrac{1}{2})[\beta\pm\ln(2-e^{-\beta})]$,
and where $\hat\ell_p$ is a slight modification of $\ell_p$ defined by
$\hat\ell_p=\ell_p-\mathds{1}_{b_p=0}$.
Furthermore, under the probability $\q_{0}^{(\beta)}$, the process $b$ can be decomposed
as follows:
\begin{enumerate}
\item[a)]
Let $\tau=(\tau_0^j)_{ j\ge 1}$ be the successive return times of $b$ at 0, and set 
$\tau_0^0=0$, $\tau_0^1=\tau_0$.
Then the sequence $\{\tau_0^j-\tau_0^{j-1};\, j\ge 1\}$ is i.i.d, and the law of $\tau_0$
is defined by its Laplace transform (\ref{new:laplace}). In  particular, $\tau_0$
has a finite mean, whose equivalent, as $\beta\to\infty$, is $1-e^{-\beta}/2$.
\item[b)]
Given the sequence $\tau$, the excursions $(b^{j})_{ j\ge 1}$, defined by
$b_r^j=b_{\tau_0^{j-1}+r}$ for $r\le \tau_0^j-\tau_0^{j-1}$, are independent.
Moreover, each $b^j$ is distributed as a random walk starting from 0, constrained
to go back to 0 at time $\tau_0^j-\tau_0^{j-1}$.
\end{enumerate}
\end{thm}

\vspace{0.3cm}

As mentioned above, the results presented in this note are not really new. In the penalization
literature, the random walk weighted by a functional of its local time has been considered
by Debs in \cite{De} for the delocalized phase, and we only cite  his result here in order
to give a complete picture of our polymer behaviour. We shall thus concentrate on the
localized phase $\beta>0$ in the remainder of the article. 
However, in this case  the results concerning
homogeneous polymers can be considered now as classical, 
and the first rigorous treatment of our
pinned model can be traced back at least to \cite{BdH}. The results we obtain for the
localized part of our theorem can also be found, in an (almost) explicit way, in 
\cite{CGZ,Gi}. But once again, our goal here is just to show that the penalization
method can be applied in this context, and may shed a new light on the polymer problem.
Furthermore, we believe that this method may be applied to other continuous
or discrete inhomogeneous models, hopefully leading to some simplifications in their analysis.
These aspects will be handled in a subsequent publication.

\vspace{0.3cm}

Let us say now a few words about the way our article is structured:
at Section \ref{sec:classical-random-walk}, we will recall some basic identities in law
for the simple symmetric random walk on $\mathds{Z}$. In order to apply our penalization
program, a fundamental step is then to get some sharp asymptotics for the semi-group
$\laa_n$ mentioned at Proposition \ref{prop:general-penalization}. This will be done
at Section \ref{sec:laplace-local-time}, thanks to the renewal trick introduced e.g. in
\cite{Gi}. This will allow to us to describe our infinite volume limit at Section \ref{sec:gibbs-limit}
in terms of the martingale $M_p^{(\beta)}$. The description of the process $b$
under the infinite volume measure given at Theorem \ref{thm:desription-gibbs-limit}
will then be proved, in terms of the behavior of $M_p^{(\beta)}$, 
at Section \ref{sec:process-new-pb-measure}.

\section{Classical facts on random walks}\label{sec:classical-random-walk}~~
\setcounter{equation}{0}
Let us first recall some basic results about the random walk $b$: for $n\ge 0$
and $z\in\mathds{Z}$, set
\begin{equation*}
S_n=\sup\{ b_p;\, p\le n \},
\quad
T_z=\inf\{n\ge 0;\, b_n=z\}
\quad\mbox{ and }\quad 
\tau_z=\inf\{n\ge 1;\, b_n=z\}.
\end{equation*}
Let us denote by $\mathds{D}$ the set of even integers in $\mathds{Z}$, and for $(n,r)\in\mathds{N}
\times\mathds{Z}$, recall that $p_{n,r}:=\p_{0}(b_n=r)$ is given by:
\begin{equation*}
p_{n,r}
=\left(\frac{1}{2}\right)^n \begin{pmatrix}\nicefrac{(n+r)}{2}\\n\end{pmatrix}
\mathds{1}_{\mathds{D}}(n+r) 
\mathds{1}_{\{ |r|\le n \}}.
\end{equation*}
Then it is well-known (see e.g. \cite{Fe,De}) that
\begin{equation}\label{distribution:Sn-Tr}
\p_{0}(S_n=r)=p_{n,r}\vee p_{n,r+1}
\quad\mbox{ and }\quad
\p_{0}(T_r=n)=
\frac{r}{n}\left(\frac{1}{2}\right)^n \begin{pmatrix}\nicefrac{(n+r)}{2}\\n\end{pmatrix}.
\end{equation}
Moreover, the distribution of $\ell_n$ can be expressed in terms of these quantities:
\begin{equation}\label{dist:loc-time}
\p_{0}(\ell_n=k)
=\p_{0}(S_{n-k}=k)+\p_{0}(T_{n+1}=n-k),
\end{equation}
and the following asymptotic results hold true:
\begin{lem}\label{equiv:sn-ln}
Let $p\in\mathds{N}$ and set $\ka=(2/\pi)^{\nicefrac{1}{2}}$. Then
\begin{equation*}
\lim_{n\to\infty}n^{\nicefrac{1}{2}}\p_{0}(S_n=p)
=\lim_{n\to\infty}n^{\nicefrac{1}{2}}\p_{0}(\ell_n=p)
=\ka,
\quad\mbox{ and }\quad
\lim_{n\to\infty}n^{\nicefrac{3}{2}}\p_{0}(T_z=n) = \ka z.\\
\end{equation*}
\end{lem}

For our further computations, we will also need the following expression for
the Laplace transform of $T_r$ and $\tau_{r}$:
\begin{lem}\label{laplace:hitt-time}
Let $r\in\mathds{N}$, $\delta>0$.  %, and set $\eta=\arg\cosh(e^\delta)$. 
Then
\begin{equation}\label{laplace:T}
\e_{0}[  e^{-\delta T_r}] = \exp\left\{-r\arg\cosh(e^\delta)\right\}
\end{equation}
and
\begin{equation}\label{laplace:tau}
\e_{0}[  e^{-\delta \tau_r}] = 
\left\{
\begin{array}{ll}
\exp\left\{-r\arg\cosh(e^\delta)\right\},&\mbox{ if }\,r\geq 1\\
\exp\left\{-\delta-\arg\cosh(e^\delta)\right\},&\mbox{ if }\,r=0
\end{array}
\right.
\end{equation}
%{\bf Ici on introduit la notation $\eta$, mais on ne s'en sert pas.}
\end{lem}
\begin{proof}
This is an elementary computation based on the fact that
$\{\exp(\eta b_n-\delta n);\,n\geq 1\}$ is a martingale.
Also, note that $\tau_{0}$ has the same law as $1+T_{1}$.
\end{proof}

\section{Laplace transform of the local time}\label{sec:laplace-local-time}~~
\setcounter{equation}{0}
Our aim in this section is to find an asymptotic equivalent for the Laplace transform 
$Z_{n}^{f}$ of $\ell_n$. However, for computational purposes, we will also have to
consider the following constrained Laplace transform :
\begin{equation*}
Z_{2m}^{c}:=\e_{0}\left[e^{\beta\ell_{2m}}\mathds{1}_{\{b_{2m}=0\}}\right]\quad(\beta\geq 0). 
\end{equation*}
With this notation in hand, here is our first result about the exponential moments of 
the local time:
\begin{lem}\label{lem:equiv-laplace:tl0}
For any $\beta>0$, we have
\begin{equation}\label{equiv:laplace:tl0c}
\lim_{m\to\infty}\left(e^{-\beta}(2-e^{-\beta})\right)^{m}Z_{2m}^{c}
= c_\beta^c,
\quad\mbox{ where }\quad
c_\beta^c:=\frac{2(1-e^{-\beta})}{2-e^{-\beta}},
\end{equation}
and
\begin{equation}\label{equiv:laplace:tl0f}
\lim_{n\to\infty}\left(e^{-\beta}(2-e^{-\beta})\right)^{\lfloor\nicefrac{n}{2}\rfloor}Z_{n}^{f}
= c_\beta^f,
\quad\mbox{ where }\quad
c_\beta^f:=\frac{2}{2-e^{-\beta}}.
\end{equation}
\end{lem}

\begin{proof}
According to (1.9)-(1.10) in \cite[p. 9]{Gi}, by using the renewal theorem, we can write
\begin{equation}\label{equa:renouv}
Z_{2m}^{c}=\e_{0}\left[e^{\beta\ell_{2m}}\mathds{1}_{\{b_{2m}=0\}}\right]
=\sum_{k=1}^{m}\sum_{{\bf r}\in A_{k,m}}
%\sum_{_{{\bf r}=(2r_{1},\ldots,2r_{k}),\sum_{j=1}^{k}r_{j}=m}}
\prod_{j=1}^{k}e^{\beta}\p_{0}(\tau_{0}=2r_{j})
%=e^{m\mbox{\small{\tt F}}(\beta)}{\tilde\p}_{\beta}(m\in\tau)\\
\underset{m\to\infty}{\sim}
\frac{e^{m\mbox{\small{\tt F}}(\beta)}}{\sum_{m}m{\tilde K}_{\beta}(m)},
\end{equation}
where we denoted $A_{k,m}=\{{\bf r}=(r_{1},\ldots,r_{k}),\sum_{j=1}^{k}r_{j}=m\}$.
%{\bf Pourquoi ne pas introduire une notation 
%$A_{r,m}=\{{\bf r}=(2r_{1},\ldots,2r_{k}),\sum_{j=1}^{k}r_{j}=m\}$?}
Here 
\begin{equation}\label{K}
{\tilde K}_{\beta}(m):=\exp\left(\beta-m\mbox{\small\tt F}(\beta)\right)K(m),
\quad\mbox{ where }\quad K(m):=\p_{0}(\tau_{0}=2m),
\end{equation}
and $\text{\small\tt F}(\beta)$ is the solution of the following equation (see also (1.6), p. 8 in 
\cite{Gi})
\begin{equation}\label{equa:F}
\sum_{m}e^{-m\mbox{\small\tt F}(\beta)}K(m)=e^{-\beta}
\quad\mbox{ i.e. }\quad
\e_0\lc e^{\nicefrac{-\mbox{\small\tt F}(\beta)\tau_0}{2}} \rc
= e^{-\beta}.
\end{equation}
Notice that in our case, equation (\ref{equa:F}) can be solved explicitly: thanks to relation
(\ref{laplace:tau}), it can be transformed into:
\begin{equation*}
\exp\left(\nicefrac{-\mbox{\small\tt F}(\beta)}{2}
-\arg\cosh\left(e^{\nicefrac{\mbox{\small\tt F}(\beta)}{2}}\right)\right)=e^{-\beta}
\Leftrightarrow
\cosh\left(\beta-\nicefrac{\mbox{\small\tt F}(\beta)}{2}\right)
=e^{\nicefrac{\mbox{\small\tt F}(\beta)}{2}}
\Leftrightarrow
e^{\beta-\mbox{\small\tt F}(\beta)}+e^{-\beta}=2,
\end{equation*}
and thus, the solution of (\ref{equa:F}) is given by
\begin{equation}\label{expr:F}
\mbox{\small\tt F}(\beta)=\beta-\ln(2-e^{-\beta}).
\end{equation}
On the other hand,
\begin{equation*}
%\varphi(\lambda)=
\sum_{m}me^{-\lambda m}\p_{0}(\tau_{0}=2m)
=-\frac{d}{d\lambda}\e_{0}\left[e^{\nicefrac{-\lambda\tau_{0}}{2}}\right]
=-\frac{d}{d\lambda}\left(1-e^{\nicefrac{-\lambda}{2}}(e^{\lambda}-1)^{\nicefrac{1}{2}}\right)
=\frac{e^{-\lambda}}{2(1-e^{-\lambda})^{\nicefrac{1}{2}}},
\end{equation*}
as we can see again by (\ref{laplace:tau}) and simple computation. Therefore, taking 
$\lambda=\mbox{\small\tt F}(\beta)$, we obtain 
\begin{equation}\label{tilde:espe}
\sum_{m}m{\tilde K}_{\beta}(m)
=e^{\beta}\sum_{m}m e^{-m\mbox{\small\tt F}(\beta)}\p_{0}(\tau_{0}=2m)
=\frac{2-e^{-\beta}}{2(1-e^{-\beta})},
\end{equation}
since, according to (\ref{expr:F}), 
$e^{-\mbox{\small\tt F}(\beta)}=e^{-\beta}(2-e^{-\beta})=1-(1-e^{-\beta})^{2}$.
Puting together (\ref{equa:renouv}), (\ref{expr:F}) and (\ref{tilde:espe}) we get 
the equivalent for the constrained Laplace transform (\ref{equiv:laplace:tl0c}).

We proceed now with the study of the free Laplace transform, called $Z_{n}^{f}$. Set 
$\overline{K}(n):=\sum_{j>n}K(j)$.  
We can write 
\begin{multline*}
Z_{2m}^{f}
=\sum_{j=0}^{m}\e_{0}\left[e^{\beta\ell_{2m}}\mathds{1}_{\max\{k\leq m,b_{2k}=0\}=j}\right]
=\sum_{j=0}^{m}\e_{0}\left[e^{\beta\ell_{2j}}\mathds{1}_{\{b_{2j}=0\}}
\mathds{1}_{\{\tau_{0}\circ\theta_{2j}>2(m-j)\}}\right]\\
=\sum_{j=0}^{m}\e_{0}\left[e^{\beta\ell_{2j}}\mathds{1}_{\{b_{2j}=0\}}\right]
\p_{0}\left(\tau_{0}>2(m-j)\right)
=\sum_{j=0}^{m}\e_{0}\left[e^{\beta\ell_{2(m-j)}}\mathds{1}_{\{b_{2(m-j)}=0\}}\right]
{\overline K}(j)\\
=\sum_{j=0}^{m}Z_{2(m-j)}^{c}{\overline K}(j)
=e^{m\mbox{\small\tt F}(\beta)}\sum_{j=0}^{m}e^{-(m-j)\mbox{\small\tt F}(\beta)}\,Z_{2(m-j)}^{c}
e^{-j\mbox{\small\tt F}(\beta)}{\overline K}(j).
\end{multline*}
In order to use (\ref{equiv:laplace:tl0c}) on the right hand side of the latter equality we need 
to apply the dominated convergence theorem. This is allowed by the inequality 
\begin{equation}\label{eq:up-bnd-zmc}
e^{-(m-j)\mbox{\small\tt F}(\beta)}\,Z_{2(m-j)}^{c}\leq 1,
\end{equation}
which is valid since $e^{-j\mbox{\small\tt F}(\beta)} Z_{2j}^{c}$ represents 
the probability that a random walk with positive increments with law ${\tilde K}_{\beta}$ 
passes by $j$ (see also (1.9) in \cite{Gi}, p. 9).
Therefore, according to (\ref{equiv:laplace:tl0c}) and (\ref{laplace:tau}),
\begin{multline*}
Z_{2m}^{f}
\underset{m\to\infty}{\sim}
c_{\beta}^{c}e^{m\mbox{\small\tt F}(\beta)}
\sum_{j=0}^{\infty}e^{-j\mbox{\small\tt F}(\beta)}\sum_{i=j+1}^{\infty}K(i)
=c_{\beta}^{c}e^{m\mbox{\small\tt F}(\beta)}
\sum_{i=1}^{\infty}K(i)\sum_{j=0}^{i-1}e^{-j\mbox{\small\tt F}(\beta)}\\
=\frac{c_{\beta}^{c}}{1-e^{-\mbox{\small\tt F}(\beta)}}e^{m\mbox{\small\tt F}(\beta)}
\left(\sum_{i=1}^{\infty}K(i)-\sum_{i=1}^{\infty}K(i)e^{-i\mbox{\small\tt F}(\beta)}\right)
=\frac{c_{\beta}^{c}e^{m\mbox{\small\tt F}(\beta)}
(1-e^{-\beta})}{1-e^{-\mbox{\small\tt F}(\beta)}}
=\frac{2e^{m\mbox{\small\tt F}(\beta)}}{2-e^{-\beta}}
\end{multline*}
and we get (\ref{equiv:laplace:tl0f}), by using (\ref{expr:F}).
To finish the proof, let us note that, for any $\beta>0$,
\begin{equation}\label{laplace:impair-free}
\e_{0}\left[e^{\beta\ell_{2m+1}}\right]=\e_{0}\left[e^{\beta\ell_{2m}}\right]=Z_{2m}^{f}
\underset{m\to\infty}{\sim}c_{\beta}^{f}e^{m\mbox{\small\tt F}(\beta)}.
\end{equation}
\end{proof}

We will now go one step further and give an equivalent of $\e_{x}\left[e^{\beta\ell_{n}}\right]$ 
for an arbitrary $x\in\mathds{Z}$. Let us denote by $\mathds{O}$ the set of odd integers in 
$\mathds{Z}$.

\begin{lem}\label{equiv:laplace:issux}
Let $x\in\mathds{Z}$ be the starting point for $b$ and recall that the constant $c_{\beta}^{f}$ 
has been defined at relation (\ref{equiv:laplace:tl0f}). Then, for any $\beta>0$, 
\begin{equation}\label{equiv:laplace:tlx}
\e_{x}\left[e^{\beta\ell_{n}}\right]
\underset{n\to\infty}{\sim}
c_{\beta}^{f}
\exp\left\{\frac{\mbox{\small\tt F}(\beta)}{2}\,
(n+|x|-\mathds{1}_{\mathds{O}}(n+x))-\beta|x|\right\}.
\end{equation}
\end{lem}

\vspace{0.05cm}

\begin{proof}
First of all, notice that, by symmetry of the random walk, 
$\e_{x}\left[e^{\beta\ell_{n}}\right]=\e_{-x}\left[e^{\beta\ell_{n}}\right]$. We will thus treat the
case of a strictly positive initial condition $x$ without loss of generality.

\vspace{0.2cm}

\noindent
\underline{Case $x,n\in\mathds{D}$.~} Let us split $\e_{x}\left[e^{\beta\ell_{2m}}\right]$ into
\begin{equation*}
\e_{x}\left[e^{\beta\ell_{2m}}\right]
=\p_{x}\left(T_0>2m\right)
+
\e_{x}\left[e^{\beta\ell_{2m}}\mathds{1}_{\{T_0\leq 2m\}}\right]
=:
D_1(2m)+D_2(2m).
\end{equation*}
Then, on the one hand,
\begin{equation*}
D_1(2m)=\p_{0}\left(T_x>2m\right)=\p_{0}\left(S_{2m}< x\right),
\end{equation*}
and thus, owing to Lemma \ref{equiv:sn-ln}, we have
\begin{equation}\label{equiv:d1n}
D_1(2m)\underset{m\to\infty}{\sim}\ka xm^{\nicefrac{-1}{2}}.
\end{equation}
On the other hand, setting $g(p)=\e_{0}\left[e^{\beta\ell_{p}}\right]$, we can write
\begin{multline*}
D_2(2m)=\e_{x}\left[\mathds{1}_{\{T_{0}\leq 2m\}}g(2m-T_{0})\right]
=\sum_{k=0}^{m}\p_{x}(T_{0}=2k)g(2(m-k))\\
=e^{m\mbox{\small\tt F}(\beta)}\sum_{k=0}^{m}\p_{x}(T_{0}=2k)e^{-k\mbox{\small\tt F}(\beta)}
g(2(m-k))e^{-(m-k)\mbox{\small\tt F}(\beta)}\\
\underset{m\to\infty}{\sim}
c_{\beta}^{f}e^{m\mbox{\small\tt F}(\beta)}
\e_{0}\left[\exp\left\{-\mbox{\small\tt F}(\beta)\frac{T_{x}}{2}\right\}\right]
=c_{\beta}^{f}\exp\left\{m\mbox{\small\tt F}(\beta)
-\arg\cosh\left(\frac{\mbox{\small\tt F}(\beta)}{2}\right)x\right\}\\
=c_{\beta}^{f}\exp\left\{\frac{\mbox{\small\tt F}(\beta)}{2}(2m+x)-\beta x\right\},
\end{multline*}
which is (\ref{equiv:laplace:tlx}). Here we used the dominated convergence theorem allowed again 
by the fact that $g(2(m-k))e^{-(m-k)\mbox{\small\tt F}(\beta)}\leq 1$
(this inequality being obtained by a little elaboration of (\ref{eq:up-bnd-zmc})).

\vspace{0.2cm}

\noindent
\underline{Case $x\in\mathds{D},n\in\mathds{O}$.~} 
Clearly, invoking the latter result, we have
\begin{equation*} 
\e_{x}\left[e^{\beta\ell_{n}}\right]=\e_{x}\left[e^{\beta\ell_{n-1}}\right]
\underset{n\to\infty}{\sim}
c_{\beta}^{f}\exp\left\{\frac{\mbox{\small\tt F}(\beta)}{2}(n-1+x)-\beta x\right\}.
\end{equation*}

\vspace{0.2cm}

\noindent
\underline{Case $x\in\mathds{O},n\in\mathds{D}$.~} 
Following a similar reasoning as for the first case, we see that it is enough to study the 
term $D_2(2m)$:
\begin{multline*}
D_2(2m)=\e_{x}\left[\mathds{1}_{\{T_{0}\leq 2m\}}g(2m-T_{0})\right]
=\sum_{k=1}^{m}\p_{x}(T_{0}=2k-1)g(2m-2k+1)\\
=\sum_{k=1}^{m}\p_{x}(T_{0}=2k-1)g(2(m-k))
\underset{m\to\infty}{\sim}
c_{\beta}^{f}e^{m\mbox{\small\tt F}(\beta)}
\sum_{k=1}^{\infty}\p_{x}(T_{0}=2k-1)e^{-k\mbox{\small\tt F}(\beta)}\\
=c_{\beta}^{f}e^{m\mbox{\small\tt F}(\beta)}
\e_{x}\left[\exp\left\{-\mbox{\small\tt F}(\beta)\frac{1+T_{0}}{2}\right\}\right]
=c_{\beta}^{f}e^{(m-\nicefrac{1}{2})\mbox{\small\tt F}(\beta)}
\e_{0}\left[\exp\left\{-\mbox{\small\tt F}(\beta)\frac{T_{x}}{2}\right\}\right]\\
=c_{\beta}^{f}e^{(m-\nicefrac{1}{2})\mbox{\small\tt F}(\beta)}
\exp\left\{\left(\frac{\mbox{\small\tt F}(\beta)}{2}-\beta\right)\right\}
=c_{\beta}^{f}\exp\left\{\frac{\mbox{\small\tt F}(\beta)}{2}(2m-1+x)-\beta x\right\}.
\end{multline*}
Here we used again the dominated convergence theorem and the fact that $\ell_{2(m-k)+1}$ 
and $\ell_{2(m-k)}$ have the same law under $\p_{0}$.

\vspace{0.2cm}

\noindent
\underline{Case $x,n\in\mathds{O}$.~} 
Again, by using the preceding result
\begin{equation*} 
\e_{x}\left[e^{\beta\ell_{n}}\right]=\e_{x}\left[e^{\beta\ell_{n+1}}\right]
\underset{n\to\infty}{\sim}
c_{\beta}^{f}\exp\left\{\frac{\mbox{\small\tt F}(\beta)}{2}(n+x)-\beta x\right\}.
\end{equation*}
\end{proof}

\section{Gibbs limit}\label{sec:gibbs-limit}~~
\setcounter{equation}{0}
Let us turn now to the asymptotic behaviour of the measure $\q_{0}^{(n,\beta)}$ defined at
(\ref{def:rho-n}). To this purpose, we will need an additional definition: for
$n\ge 0$, let ${\hat\ell}_{n}$ be the modified local time given by:
\begin{equation*}
{\hat\ell}_{n}=\ell_n-\mathds{1}_{\{b_n=0\}},
\end{equation*}
and notice that this modified local time appears here because $\ell$ satisfies the relation
\begin{equation*}
\ell_n={\hat\ell}_p+\ell_{n-p}\circ \theta_p
\quad\mbox{ instead of }\quad
\ell_n=\ell_p+\ell_{n-p}\circ \theta_p.
\end{equation*}
Indeed, it is readily checked that one zero is doubly counted in the latter relation if $b_p=0$.

With this notation in hand, the limit of $\q$ is given by the following:
\begin{prop}
For any $p\ge 0$, the measure $\q_{0}^{(n,\beta)}$ converges weakly on $\cf_{p}$, as 
$n\to\infty$, to  the measure $\q_{0}^{(\beta)}$ given by
\begin{equation}\label{new:prob}
\q_{0}^{(\beta)}(\gga_{p})=\e_{0}\left[\mathds{1}_{\gga_{p}}M_{p}^{(\beta)}\right],
\quad\mbox{ for }\quad
\gga_{p}\in\cf_{p},
\end{equation}
with $M^{(\beta)}$ a positive martingale defined by
\begin{equation}\label{martingale}
M_p^{(\beta)}=\exp\left\{\beta{\hat\ell}_{p}-c_{+}(\beta)\,|b_{p}|-c_{-}(\beta)\,p\right\},
\end{equation}
where 
\begin{equation}\label{les-constantes}
c_{\pm}(\beta)=(\nicefrac{1}{2})[\beta\pm\ln(2-e^{-\beta})].
\end{equation}
\end{prop}

\begin{proof}
For $n\ge p$, let us decompose $\ell_n$ into
$$
\ell_n={\hat\ell}_p+\ell_{n-p}\circ \theta_p.
$$
Thanks to this decomposition, we obtain, for a given $\gga_p\in\cf_p$,
\begin{equation}\label{exp2:rho-n-ga}
\q_{0}^{(n,\beta)}(\gga_p)
=\e_{0}\left[\mathds{1}_{\gga_p} e^{\beta{\hat\ell}_p} U_{n,p}(b_p)\right],
\quad\mbox{ with }\quad
U_{n,p}(x)= 
\frac{\e_x\left[e^{\beta\ell_{n-p}}\right]}{\e_0\left[e^{\beta\ell_{n}}\right]}.
\end{equation}
Moreover, according to relation (\ref{equiv:laplace:tlx}), we have, for any $x\in\mathds{Z}$,
\begin{equation}\label{equiv:unpx}
U_{n,p}(x)
\underset{n\to\infty}{\sim}
\left\{\begin{array}{ll}
\exp\{\frac{\mbox{\small\tt F}(\beta)}{2}(|x|-p)-\beta|x|-\mathds{1}_{\mathds{O}}(x+p)\}
&\mbox{ if }\,n\in\mathds{D}\\\\
\exp\{\frac{\mbox{\small\tt F}(\beta)}{2}(|x|-p)-\beta|x|+\mathds{1}_{\mathds{O}}(x+p)\}
&\mbox{ if }\,n\in\mathds{O},
\end{array}\right.
\end{equation}
where we used the symmetry on $x$. To apply the dominated convergence theorem 
let us note that 
\begin{equation*}
\e_x\left[e^{\beta\ell_{n-p}}\right]\leq \e_0\left[e^{\beta\ell_{n-p}}\right],\,\forall x\in\mathds{Z}
\quad\Rightarrow\quad
U_{n,p}(x)\leq 1,\,\forall x\in\mathds{Z}
\quad\Rightarrow\quad
U_{n,p}(b_p)\leq 1.
\end{equation*}
Therefore, we obtain that 
\begin{equation*}
M_{p}^{(\beta)}
=\exp\left\{\frac{\mbox{\small\tt F}(\beta)}{2}(b_{p}-p)-\beta b_{p}+\beta{\hat\ell}_{p}\right\},
\end{equation*}
and we deduce (\ref{martingale}).
It is now easily checked that the process $M^{(\beta)}$ is a martingale. Indeed, setting
$N_{p}^{(\beta)}=\ln(M_{p}^{(\beta)})$, and noting that $c_{+}(\beta)+c_{-}(\beta)=\beta$, we have
\begin{multline*}
N_{p+1}^{(\beta)}=\beta{\hat\ell}_{p+1}-c_{+}(\beta)\,|b_{p+1}|-c_{-}(\beta)\,(p+1)\\
=\mathds{1}_{\{b_{p}=0\}}[\beta({\hat\ell}_{p}+1)-\beta-c_{-}(\beta)\,p]
+\mathds{1}_{\{b_{p}\neq 0\}}
[\beta({\hat\ell}_{p}-c_{+}(\beta)\,(|b_{p}|+\xi_{p+1})-c_{-}(\beta)\,(p+1)],
\end{multline*}
where $\xi_{p+1}$ is a symmetric $\pm 1$-valued random variable independent of $\cf_{p}$,
representing the increment of $b$ at time $p+1$.
Hence
\begin{equation}\label{logmartingale}
N_{p+1}^{(\beta)}
=\mathds{1}_{\{b_{p}=0\}}N_{p}^{(\beta)}+\mathds{1}_{\{b_{p}\neq 0\}}[N_{p}^{(\beta)}
-c_{+}(\beta)\,\xi_{p+1}-c_{-}(\beta)].
\end{equation}
Thus
\begin{equation*}
\e_{0}[M_{p+1}^{(\beta)}\mid\cf_{p}]=\mathds{1}_{\{b_{p}=0\}}M_{p}^{(\beta)}
+\mathds{1}_{\{b_{p}\neq 0\}}M_{p}^{(\beta)}\,\cosh(c_{+}(\beta))\exp(-c_{-}(\beta)),
\end{equation*}
from which the martingale property is readily obtained from the definition (\ref{les-constantes}).
\end{proof}

\begin{rem}
It should be noticed that the convergence of $\mathds{Q}_0^{(n,\beta)}$ 
we have obtained on $\cf_p$ is stronger than the weak convergence. In fact, we have
been able to prove that, for any $\Gamma_p\in\cf_p$, we have
$\lim_{n\to\infty}\mathds{Q}_0^{(n,\beta)}(\Gamma_p)=\mathds{Q}_0^{(\beta)}(\Gamma_p)$.
This property is classical in the penalization theory.
\end{rem}

\section{The process under the new probability measure}
\label{sec:process-new-pb-measure}~~
\setcounter{equation}{0}
It must be noticed that $\q_{0}^{(\beta)}$ is a probability measure on $(\Omega,\cf,(\cf_{n})
_{n\geq 1})$, since $M_{0}^{(\beta)}=1$. In this section we study the process $\{b_{n};n\geq 1\}$ 
under the new probability measure $\q_{0}^{(\beta)}$, which recovers the results of Theorem 
\ref{thm:desription-gibbs-limit}, part 2.

\begin{prop}
Let $\q_{0}^{(\beta)}$ be the probability measure defined by {\rm (\ref{new:prob})} with 
$M^{(\beta)}$ given by {\rm (\ref{martingale})}. Then, under $\q_{0}^{(\beta)}$:
\begin{enumerate}
\item[a)] $\{b_{n};n\geq 1\}$ is a Markov process on the state space $\mathds{Z}$ 
having some transition probabilities given by 
\begin{equation}\label{new:naiss}
\q_{0}^{(\beta)}(b_{n}=r\mid b_{n-1}=r-1)=
\left\{\begin{array}{ll}
\nicefrac{e^{-\beta}}{2}&\mbox{ if }r>1\\\\
1-\nicefrac{e^{-\beta}}{2}&\mbox{ if }r<-1,
\end{array}\right.
\end{equation}
\begin{equation}\label{new:mort}
\q_{0}^{(\beta)}(b_{n}=r\mid b_{n-1}=r+1)=
\left\{\begin{array}{ll}
1-\nicefrac{e^{-\beta}}{2}&\mbox{ if }r\geq 0\\\\
\nicefrac{e^{-\beta}}{2}&\mbox{ if }r<-1
\end{array}\right.
\end{equation}
and 
\begin{equation}\label{new:0}
\q_{0}^{(\beta)}(b_{n}=1\mid b_{n-1}=0)=\q_{0}^{(\beta)}(b_{n}=-1\mid b_{n-1}=0)=\nicefrac{1}{2}.
\end{equation}
\item[b)] the Laplace transform of the first return time in 0 is given by 
\begin{equation}\label{new:laplace}
\e_{0}^{(\beta)}\left[e^{-\delta\tau_{0}}\right]
=e^{\beta}\left(e^{\delta+\mbox{\small\tt F}(\beta)}
-\left[e^{2(\delta+\mbox{\small\tt F}(\beta))}-1\right]^{\nicefrac{1}{2}}\right).
\end{equation}
In particular, $\e_{0}^{(\beta)}[\tau_0]<\infty$ for any $\beta>0$, and 
\begin{equation}\label{eq:equiv-mean-tau0}
\e_{0}^{(\beta)}[\tau_0]\sim 1-e^{-\beta/2},
\quad\mbox{ when }\quad
\beta\to\infty.
\end{equation}
\item[c)] the distribution law of the excursion between two succesive zero of the process 
 $\{b_{n};n\geq 1\}$ is the same as under $\p_{0}$. 
\end{enumerate}
\end{prop}

\begin{proof}
a) Let $\gga_{n-2}\in\cf_{n-2}$ arbitrary. Then 
\begin{multline}\label{new:markov}
\q_{0}^{(\beta)}(b_{n}=r\mid b_{n-1}=r-1,\gga_{n-2})
=\frac{\q_{0}^{(\beta)}(b_{n}=r,b_{n-1}=r-1,\gga_{n-2})}{\q_{0}^{(\beta)}(b_{n-1}=r-1,\gga_{n-2})}\\
=\frac{\e_{0}\left[\mathds{1}_{\{b_{n}=r\}}\mathds{1}_{\{b_{n-1}=r-1\}}\mathds{1}_{\gga_{n-2}}
M_{n}^{(\beta)}\right]}{\e_{0}\left[\mathds{1}_{\{b_{n-1}=r-1\}}\mathds{1}_{\gga_{n-2}}
M_{n-1}^{(\beta)}\right]}
=\frac{\e_{0}\left\{\e_{0}\left[\mathds{1}_{\{b_{n}=r\}}\mathds{1}_{\{b_{n-1}=r-1\}}\mathds{1}_{\gga_{n-2}}
M_{n}^{(\beta)}\mid\cf_{n-1}\right]\right\}}
{\e_{0}\left[\mathds{1}_{\{b_{n-1}=r-1\}}\mathds{1}_{\gga_{n-2}}M_{n-1}^{(\beta)}\right]}.
\end{multline}
First, assume that $r=1$ in the latter equality Since $M_{n}^{(\beta)}=M_{n-1}^{(\beta)}$ 
if $b_{n-1}=0$, then
\begin{multline}\label{0vers1}
\q_{0}^{(\beta)}(b_{n}=1\mid b_{n-1}=0,\gga_{n-2})
=\frac{\e_{0}\left\{\e_{0}\left[\mathds{1}_{\{b_{n}=1\}}\mathds{1}_{\{b_{n-1}=0\}}\mathds{1}_{\gga_{n-2}}
M_{n-1}^{(\beta)}\mid\cf_{n-1}\right]\right\}}
{\e_{0}\left[\mathds{1}_{\{b_{n-1}=0\}}\mathds{1}_{\gga_{n-2}}M_{n-1}^{(\beta)}\right]}\\
=\e_{0}\left[\mathds{1}_{\{b_{n}=1\}}\mid\cf_{n-1}\right]=\frac{1}{2}.
\end{multline}
The same kind of computations can be performed with $\Gamma_{n-2}=\Omega$, which gives 
\begin{equation}\label{Markov:0vers1}
\q_{0}^{(\beta)}(b_{n}=1\mid b_{n-1}=0,\gga_{n-2})
=\q_{0}^{(\beta)}(b_{n}=1\mid b_{n-1}=0).
\end{equation} 
Second, assume that $r>1$ in (\ref{new:markov}). In this case, invoking (\ref{logmartingale}) 
we have
%the fact that $|b_{n}|=|b_{n-1}|+\xi_{n}=b_{n-1}+\xi_{1}\circ\theta_{n}$ 
\begin{multline}\label{versdroite}
\q_{0}^{(\beta)}(b_{n}=r\mid b_{n-1}=r-1,\gga_{n-2})\\
=\frac{\e_{0}\left\{\e_{0}\left[\mathds{1}_{\{b_{n}=r\}}\mathds{1}_{\{b_{n-1}=r-1\}}\mathds{1}_{\gga_{n-2}}
M_{n-1}^{(\beta)}e^{-\xi_{n}c_{+}(\beta)-c_{-}(\beta)}\mid\cf_{n-1}\right]\right\}}
{\e_{0}\left[\mathds{1}_{\{b_{n-1}=r-1\}}\mathds{1}_{\gga_{n-2}}M_{n-1}^{(\beta)}\right]}\\
=\frac{\e_{0}\left\{\mathds{1}_{\{b_{n-1}=r-1\}}\mathds{1}_{\gga_{n-2}}
M_{n-1}^{(\beta)}\e_{0}\left[\mathds{1}_{\{b_{n}=r\}}e^{-\xi_{n}c_{+}(\beta)-c_{-}(\beta)}\mid\cf_{n-1}\right]
\right\}}{\e_{0}\left[\mathds{1}_{\{b_{n-1}=r-1\}}\mathds{1}_{\gga_{n-2}}M_{n-1}^{(\beta)}\right]}\\
=\e_{r-1}\left[\mathds{1}_{\{b_{1}=r\}}e^{-\xi_{1}c_{+}(\beta)-c_{-}(\beta)}\right]
=\frac{1}{2}e^{-(c_{+}(\beta)+c_{-}(\beta))}=\frac{1}{2}e^{-\beta}.
\end{multline}
Again, we can get that 
\begin{equation}\label{Markov:versdroite}
\q_{0}^{(\beta)}(b_{n}=r\mid b_{n-1}=r-1,\gga_{n-2})
=\q_{0}^{(\beta)}(b_{n}=r\mid b_{n-1}=r-1).
\end{equation}
Hence (\ref{Markov:0vers1}) and (\ref{Markov:versdroite}) prove the Markovian 
feature of the process $\{b_{n};n\geq 1\}$ under $\q_{0}^{(\beta)}$, while (\ref{0vers1}) 
and (\ref{versdroite}) prove the first equalities in 
(\ref{new:naiss}) and (\ref{new:0}). The other equalities can be obtained in a similar way.

\vspace{0.2cm}

\noindent
b) We can write 
\begin{equation*}
\q_{0}^{(\beta)}(\tau_{0}=2k)
=\e_{0}\left[\mathds{1}_{\{\tau_{0}=2k\}}M_{2k}^{(\beta)}\right]
=e^{\beta-2kc_{-}(\beta)}\p_{0}(\tau_{0}=2k)
=e^{\beta-k\mbox{\small\tt F}(\beta)}\p_{0}(\tau_{0}=2k),
\end{equation*}
where we used (\ref{martingale}) and the fact that $2c_{-}(\beta)=\mbox{\small\tt F}(\beta)$. 
Clearly, the latter equality defines a probability measure 
since, thanks to (\ref{equa:F}),
\begin{equation*}
\sum_{k\geq 1}e^{\beta-k\mbox{\small\tt F}(\beta)}\p_{0}(\tau_{0}=2k)
=e^{\beta}\e_{0}\left[e^{\nicefrac{-\mbox{\small\tt F}(\beta)\tau_{0}}{2}}\right]
=1.
\end{equation*}

Moreover, we can compute the Laplace transform of $\tau_{0}$ 
\begin{multline}\label{new:laplacetau}
\e_{0}^{(\beta)}\left[e^{-\delta\tau_{0}}\right]
=\sum_{k\geq 1}e^{-2\delta k}e^{\beta-2kc_{-}(\beta)}\p_{0}(\tau_{0}=2k)
=e^{\beta}\e_{0}\left[e^{-(\delta+\mbox{\small\tt F}(\beta))\tau_{0}}\right]\\
=\exp\left\{\beta-\arg\cosh\left(e^{\delta+\mbox{\small\tt F}(\beta)}\right)\right\}
=\frac{e^{\beta}}{e^{\delta+\mbox{\small\tt F}(\beta)}
+\left[e^{2(\delta+\mbox{\small\tt F}(\beta))}-1\right]^{\nicefrac{1}{2}}}\\
=e^{\beta}\left\{e^{\delta+\mbox{\small\tt F}(\beta)}-\left[e^{2(\delta+\mbox{\small\tt F}(\beta))}-1\right]^{\nicefrac{1}{2}}\right\}.
\end{multline}
We deduce
\begin{equation}\label{new:esperance}
\e_{0}^{(\beta)}\left[\tau_{0}\right]
=-\frac{d}{d\delta}\e_{0}^{(\beta)}\left[e^{-\delta\tau_{0}}\right]_{\mid\delta=0}
=e^{\beta+\mbox{\small\tt F}(\beta)}
\left\{\frac{1}{\left[1-e^{-2\mbox{\small\tt F}(\beta)}\right]^{\nicefrac{1}{2}}}-1\right\}.
\end{equation}
By (\ref{new:esperance}) we also get that $\lim_{\beta\to\infty}\e_{0}^{(\beta)}\left[\tau_{0}\right]
=1=\lim_{\beta\to\infty}\nicefrac{1}{\mbox{\small\tt F}'(\beta)}$, by using also (\ref{expr:F}), 
while $\e_{0}^{(\beta)}\left[\tau_{0}\right]\neq\nicefrac{1}{\mbox{\small\tt F}'(\beta)}$.
The equivalent (\ref{eq:equiv-mean-tau0}) is also easily deduced from
(\ref{new:esperance}).

\vspace{0.2cm}

\noindent
c) Thanks to the Markov property it is enough to describe the first excursion of $b$ between 0 and 
$\tau_{0}$. For any positive Borel function $f$, we have
\begin{equation*}
\e_{0}^{(\beta)}\left[f(b_{0},\ldots,b_{n})\mid\tau_{0}=n\right]
=\frac{\e_{0}\left[f(b_{0},\ldots,b_{n})\mathds{1}_{\{\tau_{0}=n\}}M_{\tau_{0}}\right]}
{\e_{0}\left[\mathds{1}_{\{\tau_{0}=n\}}M_{\tau_{0}}\right]}.
\end{equation*} 
Since, $M_{\tau_{0}}=e^{\beta-c_{-}(\beta)n}$, if $\tau_{0}=n$, we obtain that 
\begin{equation*}
\e_{0}^{(\beta)}\left[f(b_{0},\ldots,b_{n})\mid\tau_{0}=n\right]
=\e_{0}\left[f(b_{0},\ldots,b_{n})\mid\tau_{0}=n\right].
\end{equation*}
\end{proof}

\end{document}